\documentclass[twocolumn]{autart}    
   
\usepackage[cmex10]{amsmath}
\usepackage{epsfig,amsfonts,amssymb,cite,times,subfigure,floatflt,wrapfig,enumerate,tikz,url,bm}
\usepackage{algorithmicx,algpseudocode}
\usepackage{booktabs}
\usepackage{siunitx}
\usepackage{graphicx}          

\newtheorem{example}{Example}
\newtheorem{theorem}{Theorem}[section]
\newtheorem{lemma}{Lemma}[section]

\newtheorem{definition}{Definition}[section]
\newtheorem{proof}{Proof}
\newtheorem{problem}{Problem}

\newcommand{\mat}[1]{\bm{#1}}
\newcommand{\ten}[1]{\bm{\mathcal{#1}}}

\newcommand{\kpr}[1]{\textsuperscript{\textcircled{#1}}}

\begin{document}

\begin{frontmatter}

\title{Tensor Network alternating linear scheme for MIMO Volterra system identification}


\author[HKU]{Kim Batselier}\ead{kim.batselier@eee.hku.hk},    
\author[HKU]{Zhongming Chen}\ead{zmchen@eee.hku.hk},    
\author[HKU]{Ngai Wong}\ead{nwong@eee.hku.hk},               

\address[HKU]{The Department of Electrical and Electronic Engineering, The University of Hong Kong}  


\begin{abstract}                          
This article introduces two Tensor Network-based iterative algorithms for the identification of high-order discrete-time nonlinear multiple-input multiple-output (MIMO) Volterra systems. The system identification problem is rewritten in terms of a Volterra tensor, which is never explicitly constructed, thus avoiding the curse of dimensionality. It is shown how each iteration of the two identification algorithms involves solving a linear system of low computational complexity. The proposed algorithms are guaranteed to monotonically converge and numerical stability is ensured through the use of orthogonal matrix factorizations. The performance and accuracy of the two identification algorithms are illustrated by numerical experiments, where accurate degree-10 MIMO Volterra models are identified in about 1 second in Matlab on a standard desktop pc.  
\end{abstract}

\end{frontmatter}

\section{Introduction}
Volterra series~\cite{wiener2013nonlinear} have been extensively studied and applied in applications like speech modeling~\cite{Mumolo1993}, loudspeaker linearization~\cite{Kajikawa2008}, nonlinear control~\cite{fj2012identification}, active noise control~\cite{Tan2001}, modeling of biological and physiological systems~\cite{korenberg1996identification}, nonlinear communication channel identification and equalization~\cite{cheng2001optimal}, distortion analysis~\cite{Wambacq1998} and many others. Their applicability has been limited however to ``weakly nonlinear systems", where the nonlinear effects play a non-negligible role but are dominated by the linear terms. Such limitation is not inherent to the Volterra series themselves, as they can represent a wide range of nonlinear dynamical systems, but is due to the exponentially growing number of Volterra kernel coefficients as the degree increases. Indeed, assuming a finite memory $M$, the $d$th-order response of a discrete-time single-input single-output (SISO) Volterra system is given by
\begin{align*}
y_d(t) &= \sum_{k_1,\ldots,k_d=0}^{M-1} h_d(k_1,\ldots,k_d)\,\prod_{i=1}^{d} u(t-k_i), 
\end{align*}
where $y(t),u(t)$ are the scalar output and input at time $t$ respectively and the $d$th-order Volterra kernel $h_d(k_1,\ldots,k_d)$ is described by $M^d$ numbers. For a multiple-input multiple-output (MIMO) Volterra system with $p$ inputs the situation gets even worse, where the $d$th-order Volterra kernel for one particular output is characterized by $(pM)^d$ numbers. This problem is also commonly known as the \emph{curse of dimensionality}. 

The paradigm used in this article to break this curse is \textbf{to trade storage for computation}. This means that all the Volterra coefficients are replaced by only a few numbers, from which all Volterra coefficients can be computed. This idea is not new, e.g. Volterra kernels have been expanded on orthonormal basis functions in order to reduce their complexity~\cite{Campello2004,Diouf2012}. Tensors (namely, multi-dimensional arrays that are generalizations of matrices to higher orders) are also suitable candidates for this purpose. In~\cite{FKB:12} both the canonical polyadic~\cite{harshman1970fpp,candecomp} and Tucker tensor decompositions~\cite{tuckerreview} were used. The canonical polyadic decomposition can suffer from instability however, and the determination of its rank is known to be a NP-hard problem\cite{Hastad1990}, which can be ill-posed \cite{Lim2008}. The main disadvantage of the Tucker decomposition of the Volterra kernels, which is in fact an expansion onto a set of orthonormal basis functions, is that it still suffers from an exponential complexity.

This motivates us to develop and introduce a new description of discrete-time MIMO Volterra systems in terms of particular Tensor Networks (TN)~\cite{Orus2013}. For the particular case of multiple-input-single-output (MISO) systems, these Tensor Networks will turn out to be Trains (TTs)~\cite{ivanTT}. A TN representation does not suffer from any instability or exponential complexity and can represent all Volterra kernels combined by $O(d(pM+1)r^2)$ elements, where $r$ is a to-be-determined number called the TN-rank. TNs were originally developed in the physics community. Of particular importance is the Density Matrix Renormalization Group (DMRG) algorithm~\cite{White1992}, which is an iterative algorithm originally developed for the determination of the ground state of a entangled multi-body quantum system. Its applicability however is not limited to problems in quantum physics, as demonstrated by recent interest in the scientific computing community~\cite{Oseledets2011,Holtz2012,Rohwedder2013}.

In this article, we adopt the DMRG method in the TN format for the identification of MIMO discrete-time Volterra systems. The contributions of this article are twofold:
\begin{enumerate}
\item we derive a new description of discrete-time MIMO Volterra systems using the TN format,
\item we derive two iterative MIMO Volterra identification algorithms that estimate all Volterra kernels in the TN format from given input-output data.
\end{enumerate}

In each step of the iterative identification a small linear system needs to be solved. The main computational tools are the singular value decomposition (SVD) and the QR decomposition~\cite{matrixcomputations}. These orthogonal matrix factorizations ensure the numerical stability of the methods~\cite{Holtz2012}. The first identification method, which is called the Alternating Linear Scheme (ALS) method, has the lowest computational complexity but assumes that the TN-ranks $r_k$'s are fixed. The second identification method, called the Modified Alternating Linear Scheme (MALS), removes this limitation and allows for the adaptive updating of the TN-ranks during the iterations. The TN-ranks are determined numerically by means of a SVD. This is reminiscent of the determination of the order of linear systems in subspace identification algorithms~\cite{katayama2005subspace}. Monotonic convergence of both the ALS and MALS methods under certain conditions is discussed in ~\cite{Rohwedder2013}.

The outline of this article is as follows. In Section \ref{sec:prelim} we give a brief overview of important tensor concepts, operations and properties. The MIMO Volterra TN framework is introduced in Section \ref{sec:VolterraTensor}. The two iterative identification algorithms are derived in Section \ref{sec:sysid} and applied on two examples in Section \ref{sec:experiments}. To our knowledge, this is the only time where MIMO Volterra systems of degree 10 were identified in about 1 second on a standard computer. Matlab/Octave implementations of our algorithms are freely available from \url{https://github.com/kbatseli/MVMALS}.

\section{Preliminaries}
\label{sec:prelim}
\subsection{Tensor basics}
\begin{figure}[t]
\centering
\includegraphics[width=2in]{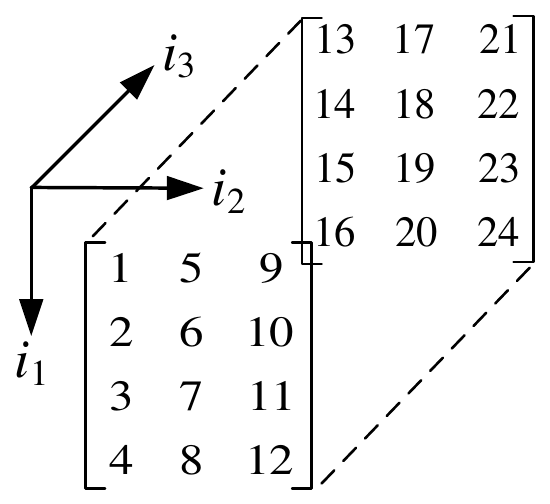}
\caption{A $4 \times 3 \times 2$ tensor where each entry is determined by three indices $i_1,i_2,i_3$.}
\label{fig:tensorexample}
\end{figure}
Tensors in this article are multi-dimensional arrays that generalize the notions of vectors and matrices to higher orders. A $d$-way or $d$th-order tensor is denoted $\ten{A} \in \mathbb{R}^{n_1 \times  n_2 \times \cdots \times n_d}$ and hence each of its entries $a_{i_1i_2\cdots i_d}$ is determined by $d$ indices. The numbers $n_1,n_2,\ldots,n_d$ are called the dimensions of the tensor. A tensor is cubical if all its dimensions are equal. A cubical tensor is symmetric when all its entries satisfy $a_{i_1i_2\cdots i_d}=a_{\pi(i_1,i_2,\ldots,i_d)}$, where $\pi(i_1,i_2,\ldots,i_d)$ is any permutation of the indices. An example $3$-way tensor with dimensions $4,3,2$ is shown in Fig.~\ref{fig:tensorexample}. For practical purposes, only real tensors are considered. We use boldface capital calligraphic letters $\ten{A},\ten{B},\ldots$ to denote tensors, boldface capital letters $\mat{A},\mat{B},\ldots$ to denote matrices, boldface letters $\mat{a},\mat{b},\ldots$ to denote vectors, and Roman letters $a,b,\ldots$ to denote scalars. The transpose of a matrix $\mat{A}$ or vector $\mat{a}$ are denoted by $\mat{A}^T$ and $\mat{a}^T$, respectively. The unit matrix of order $n$ is denoted $\mat{I}_n$.

We now give a brief description of some required tensor operations and properties. The generalization of the matrix-matrix multiplication to tensors involves a multiplication of a matrix with a $d$-way tensor along one of its $d$ possible modes (equiv. indices/axes).
\begin{definition}{\textbf{(Tensor $\bm{k}$-mode product)}}(~\cite[p.~460]{tensorreview})
The $k$-mode product $\ten{B}=\ten{A}\, {\times_k}\, \bm{U}$ of a tensor $\ten{A}\in\mathbb{R}^{n_1\times\cdots \times n_k\times\cdots\times n_d}$ with a matrix $\bm{U}\in\mathbb{R}^{p_k\times n_k}$ is defined by%
\begin{align}
b_{i_1\cdots i_{k-1} j i_{k+1} \cdots i_d}=\sum\limits_{i_k=1}^{n_k}  u_{j i_k} a_{i_1\cdots i_{k-1} i_k i_{k+1}\cdots i_d},%
\label{eqn:kmode}
\end{align}%
and $\ten{B}\in\mathbb{R}^{n_1\times\cdots \times n_{k-1}\times p_k\times n_{k+1}\times\cdots\times n_d}$.
\end{definition}
The following illustrative example rewrites the familiar matrix multiplication as a 1-mode and 2-mode product.
\begin{example}{\textbf{(Matrix multiplication as mode products)}}
For matrices $\bm{A},\bm{B},\bm{C}$ with matching dimensions we have that
\begin{align*}
\bm{A} \times_1 \bm{B} \times_2 \bm{C} &:= \bm{B}\,\bm{A}\, \bm{C}^T.
\end{align*}
\end{example}
An interesting observation is that the definition of the $k$-mode product also includes the multiplication of a tensor $\ten{A}$ with $d$ vectors. The following example highlights a very important case.
\begin{example}
Consider a symmetric $d$-way tensor $\ten{A}$ with dimensions $n$ and a vector $\mat{x} \in \mathbb{R}^{n}$. The multidimensional contraction of  $\ten{A}$ with $\mat{x}=\begin{pmatrix}x_1 & x_2 &\cdots & x_n\end{pmatrix}^T$ is the scalar
\begin{align}
\ten{A} \, \mat{x}^d &:= \ten{A} \times_1 \mat{x}^T \times_2 \mat{x}^T \times_3 \cdots \times_d \mat{x}^T,
\label{eq:symtenhompoly}
\end{align}
which is obtained as a homogeneous polynomial of degree $d$ in the variables $x_1,\ldots,x_n$.
\end{example}

The Kronecker product plays a crucial role in our description of MIMO Volterra systems.
\begin{definition}{\textbf{(Kronecker product)}}(~\cite[p.~461]{tensorreview})
If $\mat{B} \in \mathbb{R}^{m_1 \times m_2}$ and $\mat{C} \in \mathbb{R}^{n_1 \times n_2}$, then their Kronecker product $\mat{B} \otimes \mat{C}$ is an $m_1 \times m_2$ block matrix whose $(i_3,i_4)$th block is the $n_1 \times n_2$ matrix $b_{i_3i_4}\mat{C}$
\begin{equation}
\begin{pmatrix}
b_{11} & \cdots & b_{1n_1}\\
\vdots & \ddots & \vdots \\
b_{m_11} & \cdots &b_{m_1n_1}\\
\end{pmatrix} \otimes \mat{C} \;=\;
\begin{pmatrix}
b_{11}\mat{C} & \cdots & b_{1n_1}\mat{C}\\
\vdots & \ddots & \vdots \\
b_{m_11}\mat{C} & \cdots & b_{m_1n_1}\mat{C}\\
\end{pmatrix}.
\label{def:kron}
\end{equation}
\end{definition}
We use the notation $\mat{x}\kpr{d}:= \mat{x}\otimes \mat{x}\otimes \cdots \otimes \mat{x}$ for the $d$-times repeated Kronecker product. The mixed-product property of the Kronecker product states that if $\mat{A},\mat{B}, \mat{C}$ and $\mat{D}$ are matrices of such sizes that one can form the matrix products $\mat{AC}$ and $\mat{BD}$, then
\begin{align}
(\mat{A} \otimes \mat{B})\,(\mat{C} \otimes \mat{D}) = \mat{AC} \otimes \mat{BD}.
\label{eq:mixedproduct}
 \end{align}

\begin{definition}{\textbf{(Reshaping)}}(~\cite[p.~460]{tensorreview}) Reshaping is another often used tensor operation. The most common reshaping is the matricization, which reorders the entries of $\ten{A}$ into a matrix. We adopt the Matlab/Octave reshape operator ``reshape($\ten{A},[n_1,n_2,n_3 \cdots])$", which reshapes the tensor $\ten{A}$ into a tensor with dimensions $n_1,n_2,n_3,\ldots$. The total number of elements of $\ten{A}$ must be the same as $n_1\times n_2 \times n_3 \cdots$.
\end{definition}

\begin{example}
We illustrate the reshaping operator on the $4\times 3 \times 2$ tensor of Fig.~\ref{fig:tensorexample}
\begin{align*}
\textrm{reshape}(\ten{A},[4,6]) &= 
\begin{pmatrix}
1 & 5 & 9 & 13 & 17 & 21\\
2 & 6 & 10 & 14 &18 & 22\\
3 & 7 & 11 &15 &19 & 23\\
4 & 8 & 12 & 16&20 & 24
\end{pmatrix}.
\end{align*}
\end{example}


Probably the most important reshaping of a tensor is the vectorization.
\begin{definition}{\textbf{(Vectorization)}}(~\cite[p.~460]{tensorreview})
The vectorization of a tensor $\ten{A}$, denoted $\textrm{vec}(\ten{A})$, rearranges all its entries into one column vector. 
\end{definition}

\begin{example}
For the tensor in Fig.~\ref{fig:tensorexample}, we have 
\begin{align*}
\textrm{vec}(\ten{A}) \; =\textrm{reshape}(\ten{A},[24,1])=\; \begin{pmatrix}
1& 2& \cdots &24
\end{pmatrix}^T.
\end{align*}
\end{example}
The importance of the vectorization lies in the following equation
\begin{align}
\textrm{vec}(\ten{A} \times_1 \mat{U}_1 \times_2  \cdots \times_d \mat{U}_d) &= (\mat{U}_d \otimes \cdots \otimes \mat{U}_1)\textrm{vec}(\ten{A}).
\label{eq:allmodesmatrices}
\end{align}
Observe how the order in the Kronecker product is reversed with respect to the ordering of the mode products. Equation \eqref{eq:allmodesmatrices} allows us to rewrite \eqref{eq:symtenhompoly} as
\begin{align}
\ten{A} \, \mat{x}^{d} &= \textrm{vec}(\ten{A})^T \bm{x}\kpr{d},
\label{eq:allmodes}
\end{align}
which tells us how the $k$-mode products of a tensor $\ten{A}$ with a vector $\bm{x}$ can be computed. When $\ten{A}$ in \eqref{eq:allmodesmatrices} is also a matrix, we then obtain the following useful property
\begin{align}
\textrm{vec}(\mat{U}_1\,\mat{A}\,\mat{U}_2^T) &= (\mat{U}_2 \otimes \mat{U}_1)\, \textrm{vec}(\mat{A}).
\label{eq:vecofmatrixproduct}
\end{align}

\subsection{Tensor Train decomposition}
\begin{figure}[t]
\centering
\includegraphics[width=3in]{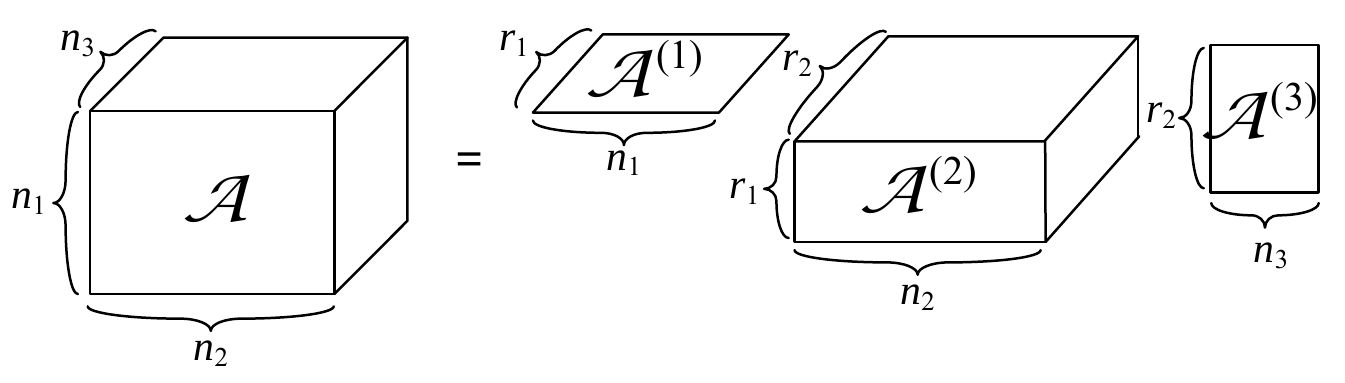}
\caption{The TT-cores of a 3-way tensor $\ten{A}$ are two matrices $\ten{A}^{(1)},\ten{A}^{(3)}$ and a $3$-way tensor $\ten{A}^{(2)}$.}
\label{fig:tt}
\end{figure}
The Tensor Network representation used in this article lies conceptually very close to the Tensor Train decomposition~\cite{ivanTT}. We therefore first discuss the Tensor Train representation. A TT-decomposition represents a $d$-way tensor $\ten{A}$ in terms of $d$ 3-way tensors $\ten{A}^{(1)}, \ldots,\ten{A}^{(d)}$, which are called the TT-cores. The $k$th TT-core has dimensions $r_{k-1},n_k,r_k$, where $r_{k-1},r_k$ are called the TT-ranks. Specifically, each entry of $\ten{A} \in \mathbb{R}^{n_1\times \cdots \times n_d}$ is determined by
\begin{equation}
a_{i_1 i_2 \cdots i_d} = \ten{A}^{(1)}_{i_1} \, \ten{A}^{(2)}_{i_2} \cdots \ten{A}^{(d)}_{i_d},
\label{eq:TTdecomp}
\end{equation}
where $\ten{A}^{(k)}_{i_k}$ is the $r_{k-1} \times r_k$ matrix obtained from specifying $i_k$. Since $a_{i_1 i_2 \cdots i_d}$ is a scalar, it immediately follows that $r_0=r_d=1$. The TT-decomposition is illustrated for a $3$-way tensor in Fig.~\ref{fig:tt}. Note that when all TT-ranks are equal to $r$, then the storage of a cubical $d$-way tensor with dimensions $n$ in the TT format needs $O(dnr^2)$ elements. Small TT-ranks therefore result in a significant reduction of required storage cost. In order to describe MIMO Volterra systems we will need to remove the constraint that $r_0=1$ and therefore obtain a slightly more general TN. Observe that the removal of this constraint implies that the left-hand side of \eqref{eq:TTdecomp} will then be a $r_0 \times 1$ vector. The following multidimensional contraction in the TN format turns out to be very important for Volterra systems.
\begin{lemma}(~\cite[p.~2309]{ivanTT})
Given a cubical tensor $\ten{A} \in \mathbb{R}^{n \times \cdots \times n}$ in the TT format and a vector $\mat{x} \in \mathbb{R}^n$. The multidimensional contraction $\ten{A} \, \mat{x}^{d}$ can then be computed as
\begin{align}
\ten{A} \mat{x}^{d}=(\ten{A}^{(1)} \times_2 \mat{x}^T) \, (\ten{A}^{(2)} \times_2 \mat{x}^T)  \cdots (\ten{A}^{(d)} \times_2 \mat{x}^T).
\label{eq:vecmatrixproduct}
\end{align}
\label{lemma:TTcontract}
 \end{lemma}
 Again, in the MIMO case this contraction will result in a $r_0 \times 1$ vector where now each entry is the evaluation of a homogeneous polynomial. The numerical stability of the two identification algorithms described in this article relies on the TN-cores being either left or right orthogonal.
 \begin{definition}(\textbf{Left orthogonal and right orthogonal TN-cores})(~\cite[p.~A689]{Holtz2012})
 A TN-core $\ten{A}^{(i)}$ is left orthogonal if it can be reshaped into an $r_{i-1}n_i \times r_i$ matrix $\mat{A}$ for which
 \begin{align*}
 \mat{A}^T \, \mat{A} &= I_{r_i}
\end{align*}
applies. Similarly, a TN-core $\ten{A}^{(i)}$ is right orthogonal if it can be reshaped into an $r_{i-1}\times n_i r_i$ matrix $\mat{A}$ for which
 \begin{align*}
 \mat{A} \, \mat{A}^T &= I_{r_{i-1}}
\end{align*}
applies.
\end{definition}


\section{Tensor description of MIMO Volterra systems}
\label{sec:VolterraTensor}
To keep the notation simple, we first consider the following discrete-time SISO Volterra system of degree $d$, namely,
\begin{align}
\nonumber y(t) &= y_0(t) + y_1(t) + y_2(t) + \cdots + y_d(t),\\
&= h_0 + \sum_{i=1}^d \sum_{k_1,\ldots,k_i=0}^{M-1} h_i(k_1,\ldots,k_i)\,\prod_{j=1}^{i} u(t-k_j), 
\label{eq:SISOVolterra}
\end{align}
where $y(t),u(t)$ are the scalar output and input at time $t$ respectively, $M$ is the memory length and $h_i(k_1,\ldots,k_i)$ denotes the $i$th Volterra kernel. Previous work that uses tensors in Volterra series describes each Volterra kernel as a separate symmetric tensor~\cite{FKB:12}. Realizing from \eqref{eq:SISOVolterra} that $y(t)$ is a multivariate polynomial in $u(t),\ldots,u(t-M+1)$, we define the SISO Volterra tensor as follows.
\begin{definition} (SISO Volterra tensor)
Given a discrete-time SISO Volterra system of degree $d$ and memory $M$ as described in \eqref{eq:SISOVolterra}, then the $d$-way cubical Volterra tensor $\ten{V}$ of dimension $M+1$ is defined by
\begin{align*}
y(t) &= \ten{V} \, \mat{u}_t^d,
\end{align*}
where
\begin{align*}
\mat{u}_t &:= \begin{pmatrix}1 & u(t) & u(t-1) & \cdots & u(t-M+1) \end{pmatrix}^T \in \mathbb{R}^{M+1}.
\end{align*}
\label{def:SISOVolterraTensor}
\end{definition}
The extension of this definition to the MIMO case is straightforward. Indeed, the $i$th output $y_i(t)$ is then a multivariate polynomial of degree $d$ in $u_1(t),\ldots,u_1(t-M+1),u_2(t),\ldots,u_2(t-M+1),\ldots,u_p(t-M+1)$. This implies that the vector $\mat{u}_t$ in Definition~\ref{def:SISOVolterraTensor} needs to be extended with these additional inputs, which results in an increase in the dimension of the corresponding Volterra tensor. In addition, the order of the Volterra tensor is incremented to accommodate for the multiple number of outputs. 
\begin{definition} (MIMO Volterra tensor)
Given a discrete time $p$-input $l$-output Volterra system of degree $d$ and memory $M$, then the $d+1$-way Volterra tensor $\ten{V}$ of dimensions $l \times pM+1 \times \cdots \times pM+1$ is defined by
\begin{align}
\mat{y}(t) &:= \begin{pmatrix} y_1(t) \\ y_2(t) \\ \vdots \\ y_l(t) \end{pmatrix} = \ten{V} \times_2 \mat{u}_t^T \times_3 \mat{u}_t^T \cdots \times_{d+1} \mat{u}_t^T,
\label{eq:MIMOVolterraTensor}
\end{align}
where $\mat{u}_t$ is the $(pM+1)\times 1$ vector with entries $1,u_1(t),u_2(t),\ldots,u_p(t),\ldots, u_1(t-M+1),u_2(t-M+1),\ldots,u_p(t-M+1)$.
\label{def:MIMOVolterraTensor}
\end{definition}

The MIMO Volterra tensor consists of $l\,(pM+1)^d$ entries, which can quickly become practically infeasible to store as $M$ and $d$ grow. We therefore propose to store the Volterra tensor $\ten{V}$ in the TN format $\ten{V}^{(1)},\ldots,\ten{V}^{(d)}$, where the first TN-core $\ten{V}^{(1)}$ has dimensions $l \times pM+1 \times r_1$. The other TN-cores have sizes $r_{i-1} \times pM+1 \times r_i$, with $r_d=1$ for the last core. The TN format reduces to a TT for the MISO case $(l=1)$. This change in representation reduces the storage requirement to $O((d-1)(pM+1)r^2+(pM+1)lr)$. Lemma \ref{lemma:TTcontract} then immediately tells us how to simulate the output samples at time $t$ as
\begin{align}
\mat{y}(t) &= (\ten{V}^{(1)} \times_2 \mat{u}_t^T) \, (\ten{V}^{(2)} \times_2 \mat{u}_t^T)  \cdots (\ten{V}^{(d)} \times_2 \mat{u}_t^T),
\label{eq:simTT}
\end{align}
with a computational complexity of $O(d(pM+1)r+dr^3)$. In \cite{batselier2016iccad} a much faster simulation complexity of $O(dRN\log{N})$ for computing $N$ samples is obtained by using a symmetric polyadic representation for each of the Volterra kernels separately. However, the method only works for the SISO case and the identification of such a representation can be problematic, since computation of the canonical rank $R$ is an NP-hard problem\cite{Hastad1990}, which can be ill-posed \cite{Lim2008}.

\section{MIMO Volterra system identification}
\label{sec:sysid}
In what follows we always consider the MIMO case. Observe now that we can rewrite \eqref{eq:MIMOVolterraTensor} as
\begin{align}
\mat{y}(t)^T &= (\mat{u}_t\kpr{d})^T \, (\ten{V}_{(1)})^T,
\label{eq:onerow}
\end{align}
where $\ten{V}_{(1)}$ is the Volterra tensor reshaped into a $l \times (pM+1)^d$ matrix. Writing out \eqref{eq:onerow} for $t=0,1,\ldots,N-1$ leads to the following matrix equation
\begin{align}
\mat{Y} &= \mat{U}\, (\ten{V}_{(1)})^T,
\label{eq:bigsystem}
\end{align}
where
\begin{align*}
\mat{Y} &:= \begin{pmatrix} \mat{y}(0) & \mat{y}(1) & \mat{y}(2) & \cdots & \mat{y}(N-1) \end{pmatrix}^T,\\
\mat{U} &:= \begin{pmatrix} \mat{u}_0\kpr{d} & \mat{u}_1\kpr{d} & \mat{u}_2\kpr{d} & \cdots & \mat{u}_{N-1}\kpr{d} \end{pmatrix}^T.
\end{align*}
Essentially, MIMO Volterra system identification is in its most basic form solving the matrix equation \eqref{eq:bigsystem}. This immediately reveals the difficulty, as the size of the $\mat{U}$ matrix is $N \times (pM+1)^d$, which quickly becomes prohibitive even for SISO systems and moderate degree $d$. This motivates us to solve the following problem.
\begin{problem}
For a given set of measured time series $y_1,\ldots,y_l,u_1,\ldots,u_p$, memory $M$ and degree $d$, solve the 
matrix equation \eqref{eq:bigsystem} for $\ten{V}$ in the TN format.
\label{problem:sysidproblem}
\end{problem}
Before presenting the two numerical algorithms that solve Problem \ref{problem:sysidproblem}, we first give an upper bound on the rank of $\mat{U}$, together with some important implications.
\begin{lemma}
The rank of the matrix $\mat{U}$ is upper bounded by ${pM+d \choose pM}$.
\label{lemma:rankU}
\end{lemma}
\begin{proof}
Each row of $\mat{U}$ consists of $(pM+1)^d$ entries formed by the repeated Kronecker product $\mat{u}_k\kpr{d}$. There are however only ${pM+d \choose pM}$ distinct entries.
\end{proof}
Lemma \ref{lemma:rankU} motivates us to define a particular set of inputs such that this upper bound is achieved.
\begin{definition}
A set of input signals $u_1,u_2,\ldots,u_p$ such that
\begin{align*}
\textrm{rank}(\mat{U}) &= {pM+d \choose pM}
\end{align*}
is called a set of persistent exciting inputs of order $d$.
\end{definition}
We will from here on always assume that the inputs are persistent exciting and that $N \geq \textrm{rank}(\mat{U})$. The truncated Volterra series for one particular output is a polynomial in $pM$ variables of degree $d$ and therefore has ${pM+d \choose pM}$ coefficients. A solution $\ten{V}_{(1)}$ of the 
matrix equation \eqref{eq:bigsystem} that consists of only $l{pM+d \choose pM}$ distinct entries therefore must correspond with $l$ symmetric tensors, where each column of $\ten{V}_{(1)}$ corresponds with one particular $d$-way symmetric tensor. As a result, persistent exciting inputs and sufficient number of samples $N$ together with Lemma \ref{lemma:rankU} has the following consequences:
\begin{enumerate}
\item the identification problem \eqref{eq:bigsystem} has an infinite number of solutions,
\item the $l$ unique minimal norm solutions of \eqref{eq:bigsystem} correspond with $l$ symmetric tensors,
\item the unicity of the symmetric solutions implies that no other symmetric solutions exist,
\item any solution $\ten{V}_{(1)}$ of \eqref{eq:bigsystem} can be turned into $l$ minimal norm solutions by symmetrizing each of the tensors given by the columns of $\ten{V}_{(1)}$.
\end{enumerate}
Ideally, one would solve \eqref{eq:bigsystem} for the minimal norm solutions, which acts as a regularization of the problem such that the solution is uniquely defined. The two iterative methods which we describe in the next subsections do not guarantee convergence to the minimal norm solutions. However, in practice we observe that the norms of the obtained solutions are quite close to the minimal ones.

\subsection{Alternating Linear Scheme method}
The first method that we derive is the Alternating Linear Scheme (ALS) method. The key idea of this method is to fix the TN-ranks $r_1,r_2,\ldots,r_{d-1}$ and choose a particular initial guess for $\ten{V}^{(1)},\ldots,\ten{V}^{(d)}$. Each of the TN-cores is then updated separately in an iterative fashion until convergence has been reached. Once the core $\ten{V}^{(d)}$ is updated, the algorithm ``sweeps" back towards $\ten{V}^{(1)}$ and so on. It turns out that updating one of the cores is equivalent with solving a much smaller linear system, which can be done very efficiently. Before describing the form of the reduced linear system, we first introduce the following notation
\begin{align*}
\mat{v}_{k-1} &:= (\ten{V}^{(1)} \times_2 \mat{u}_t^T)  \cdots (\ten{V}^{(k-1)} \times_2 \mat{u}_t^T) \in \mathbb{R}^{l \times r_{k-1}},\\
\mat{v}_{k+1} &:= (\ten{V}^{(k+1)} \times_2 \mat{u}_t^T)  \cdots (\ten{V}^{(d)} \times_2 \mat{u}_t^T) \in \mathbb{R}^{r_k \times 1}.
\end{align*}
\begin{theorem}
For outputs $y_1,\ldots,y_l$ described by a MIMO Volterra system \eqref{eq:simTT} we have that
\begin{align}
\mat{y}(t) &= (\mat{v}_{k+1} ^T \otimes \mat{u}_t^T \otimes \mat{v}_{k-1} )\, \textrm{vec}(\ten{V}^{(k)}).
\label{eq:ALSrow}
\end{align}
\label{theo:ALSreduced}
\end{theorem}
\begin{proof}
We first rewrite \eqref{eq:simTT} as
\begin{align*}
\mat{y}(t)&= \mat{v}_{k-1}  \, (\ten{V}^{(k)} \times_2 \mat{u}_t^T) \, \mat{v}_{k+1}.
\end{align*}
This equation holds since it is the product of the $l \times r_{k-1}$ matrix $\mat{v}_{k-1} $ with the $r_{k-1} \times r_k$ matrix $(\ten{V}^{(k)} \times_2 \mat{u}_t^T)$ with the $r_k \times 1$ vector $\mat{v}_{k+1} $, resulting in the $l \times 1$ vector $\mat{y}(t)$. We then have that
\begin{align*}
\mat{y}(t)&= \mat{v}_{k-1}  \; (\ten{V}^{(k)} \times_2 \mat{u}_t^T) \; \mat{v}_{k+1},\\
&= (\mat{v}_{k+1}^T \otimes \mat{v}_{k-1} ) \; \textrm{vec}(\ten{V}^{(k)} \times_2 \mat{u}_t^T),\\
&= (\mat{v}_{k+1}^T \otimes \mat{v}_{k-1} ) \; \textrm{vec}(\ten{V}^{(k)} \times_1 \mat{I}_{r_{k-1}} \times_2 \mat{u}_t^T \times_3 \mat{I}_{r_k}),\\
&= (\mat{v}_{k+1}^T \otimes \mat{v}_{k-1} )\; (\mat{I}_{r_k} \otimes \mat{u}_t^T \otimes \mat{I}_{r_{k-1}})\; \textrm{vec}(\ten{V}^{(k)}),\\
&= (\mat{v}_{k+1}^T \otimes 1 \otimes  \mat{v}_{k-1} )\; (\mat{I}_{r_k} \otimes \mat{u}_t^T \otimes \mat{I}_{r_{k-1}})\; \textrm{vec}(\ten{V}^{(k)}),\\
&= (\mat{v}_{k+1}^T \otimes \mat{u}_t^T \otimes \mat{v}_{k-1} )\; \textrm{vec}(\ten{V}^{(k)}),
\end{align*}
where the second equation is obtained from \eqref{eq:vecofmatrixproduct}, the fourth equation is obtained from using \eqref{eq:allmodesmatrices} and the final equation follows from \eqref{eq:mixedproduct}. This concludes the proof.
\end{proof}
The importance of Theorem \ref{theo:ALSreduced} lies in the fact that it describes how the reduced linear system to update $\ten{V}^{(k)}$ can be constructed. Indeed, suppose we want to update $\ten{V}^{(k)}$ and keep all other TN-cores fixed. By using \eqref{eq:ALSrow} for $t=0,\ldots,N-1$ the following linear system is obtained
\begin{align}
\textrm{vec}(\mat{Y}^T) &= \mat{U}_k \; \textrm{vec}(\ten{V}^{(k)}),
\label{eq:ALSsystem}
\end{align}
where $\mat{U}_k$ is a $lN \times r_{k-1}(pM+1)r_k$ matrix. Next to the obvious savings in computational complexity compared to solving \eqref{eq:bigsystem}, this effectively requires much fewer samples $N$ to perform the identification. Indeed, one only has to make sure that $lN \geq r_{k-1}(pM+1)r_k$. Computing the minimal norm solution of \eqref{eq:ALSsystem} requires the computation of the pseudo-inverse of $\mat{U}_k$, which requires $O(4Nl(r_{k-1}(pM+1)r_k)^2+8(r_{k-1}(pM+1)r_k)^3)$ computations.

Numerical stability of the ALS algorithm is guaranteed by the use of an orthogonalization step~\cite[p.~A690]{Holtz2012}. The key idea is that all TN-cores are initialized to be right orthogonal and are kept orthogonal during each step. After updating $\ten{V}^{(1)}$ by solving \eqref{eq:ALSsystem}, this tensor is reshaped into a $r_0(pM+1)\times r_1$ matrix from which a thin QR decomposition~\cite[p.~230]{matrixcomputations} is computed. This takes $O(2l(pM+1)r_1^2+2r_1^3/2)$ flops. The orthogonal matrix $\mat{Q}$ is then chosen as a new left orthogonal $\ten{V}^{(1)}$ TN-core. The corresponding $\mat{R}$ matrix is then absorbed into the $\ten{V}^{(2)}$ core by reshaping the core into a $r_1\times (pM+1)r_2$ matrix $\mat{V}_2$ and computing $\mat{R}\,\mat{V}_2$. Next, $\ten{V}^{(2)}$ is updated and orthogonalized, after which $\ten{V}^{(3)}$ is updated and so forth. When the iterative ``sweep" has reached the end of the TN, it reverses direction and in a similar way all cores are made right orthogonal by  QR decomposition. The whole algorithm is presented as pseudocode in Algorithm \ref{alg:ALS}. A Matlab/Octave implementation of Algorithm \ref{alg:ALS} can be freely downloaded from \url{https://github.com/kbatseli/MVMALS}.

\begin{alg}MIMO Volterra ALS Identification\\
\label{alg:ALS}
\textit{\textbf{Input}}: $N$ samples of inputs $u_1,\ldots,u_p$, outputs $y_1,\ldots,y_l$, degree $d$, memory $M$, TN ranks $r_1,\ldots,r_{d-1}$\\
\textit{\textbf{Output}}:\makebox[0pt][l]{ TN-cores $\ten{V}^{(1)},\ldots,\ten{V}^{(d)}$ that solve Problem \ref{problem:sysidproblem}}
\begin{algorithmic}
\State Initialize right orthogonal TN-cores $\ten{V}^{(1)},\ldots,\ten{V}^{(d)}$ of prescribed ranks
\While{termination criterion not satisfied}
\For{$i=1,\ldots,d-1$}
\State $\textrm{vec}(\ten{V}^{(i)}) \gets$ Compute and solve \eqref{eq:ALSsystem}
\State $\mat{V}_i \gets$ reshape($\ten{V}^{(i)},[r_{i-1}, (pM+1)r_i])$
\State Compute thin QR decomposition of $\mat{V}_i$
\State $\ten{V}^{(i)} \gets$ reshape($\mat{Q},[r_{i-1}, (pM+1), r_i])$
\State $\mat{V}_{i+1} \gets$ reshape($\ten{V}^{(i+1)},r_i,(pM+1)r_{i+1}])$
\State $\ten{V}^{(i+1)} \gets$ reshape($\mat{R}\mat{V}_{i+1},[r_{i}, (pM+1), r_{i+1}])$
\EndFor
\State Repeat the above loop in the reverse order
\EndWhile
\end{algorithmic}
\end{alg}
A few remarks on the ALS method are in order.
\begin{itemize}
\item The iterations can be terminated when the solution does not exhibit any further improvement or when a certain fixed maximal number of sweeps has been executed. In our implementation we set a tolerance on the relative residual $||\mat{Y}-\mat{\hat{Y}} ||_2/||\mat{Y}||_2$, where $\mat{\hat{Y}}$ is the simulated output from the obtained solution.
\item It is proved in \cite{Holtz2012,Rohwedder2013} that under certain conditions the ALS method enjoys strictly monotonous convergence. Convergence to the unique minimal norm solutions cannot be guaranteed however.
\item Also in \cite[p.~A701]{Holtz2012}, it is proved that the QR decomposition step ensures that the condition number of each $\mat{U}_k$ matrix in \eqref{eq:ALSsystem} is upper bounded by the condition number of the large $\mat{U}$ matrix. This ensures the numerical stability of the ALS method.
\end{itemize}

While the ALS method has a small computational complexity, it suffers from the problem that all TN-ranks need to be specified \textit{a priori}. Finding a good choice can be quite difficult when $d$ is large. This is the main motivation for the development of the Modified Alternating Linear Scheme (MALS) algorithm, which not only updates the TN-cores but also adapts the TN-ranks during each step. An additional benefit is that the TN-cores are guaranteed to be either left or right orthogonal so that no stabilizing QR step is required anymore. These benefits, however, come at the cost of a higher computational complexity.

\subsection{Modified Alternating Linear Scheme method}
The main idea of the MALS is in fact a simple one, namely, to update two TN-cores $\ten{V}^{(k)},\ten{V}^{(k+1)}$ at a time by considering them as one ``super-core" and keeping all other cores fixed. The updated super-core is then decomposed into one orthogonal and one non-orthogonal part, which are used as updates for both $\ten{V}^{(k)}$ and $\ten{V}^{(k+1)}$. This decomposition step is achieved by computing a singular value decomposition (SVD) and it is here where the TN-rank $r_k$ is updated. This updating procedure is repeated in the same sweeping fashion as with the ALS method until the solution has converged. We now derive the reduced linear system for computing a new super-core. The entries of the super-core $\ten{V}^{(k,k+1)}$ are defined as the contraction of $\ten{V}^{(k)}$ with $\ten{V}^{(k+1)}$ by summing over all possible $r_k$ values
\begin{align*}
\ten{V}^{(k,k+1)}_{i_1[i_2i_3]i_4} &= \sum_{j=1}^{r_k} \ten{V}^{(k)}_{i_1i_2j} \ten{V}^{(k+1)}_{ji_3i_4}.
\end{align*}
The square brackets around $i_2i_3$ indicate that these indices need to be interpreted as one single multi-index such that $\ten{V}^{(k,k+1)}$ is an $r_{k-1} \times (pM+1)^2 \times r_{k+1}$ 3-way tensor. Observe that the formation of the super-core has removed all information on $r_k$. It is now straightforward to verify that
\begin{align*}
(\ten{V}^{(k)} \times_2 \mat{u}_t^T) (\ten{V}^{(k+1)} \times_2 \mat{u}_t^T)  &= \ten{V}^{(k,k+1)} \times_2 (\mat{u}_t^T)\kpr{2}
\end{align*}
holds. Using the same notation as in the ALS method we now derive the reduced linear system for the MALS.
\begin{theorem}
For outputs $y_1,\ldots,y_l$ described by a MIMO Volterra system \eqref{eq:simTT} we have that
\begin{align}
\mat{y}(t) &= (\mat{v}_{k+2} ^T \otimes (\mat{u}_t^T)\kpr{2} \otimes \mat{v}_{k-1} )\, \textrm{vec}(\ten{V}^{(k,k+1)}).
\label{eq:MALSrow}
\end{align}
\label{theo:MALSreduced}
\end{theorem}
\begin{proof}
By rewriting \eqref{eq:simTT} this time as
\begin{align*}
\mat{y}(t)&= \mat{v}_{k-1}  \, (\ten{V}^{(k,k+1)} \times_2 (\mat{u}_t^T)\kpr{2}) \, \mat{v}_{k+2},
\end{align*}
it can be clearly seen that the rest of the proof is now identical to the ALS case.
\end{proof}
Application of Theorem \ref{theo:MALSreduced} for $t=0,\ldots,N-1$ now results in the reduced linear system
\begin{align}
\textrm{vec}(\mat{Y}^T) &= \mat{U}_{k,k+1}\; \textrm{vec}(\ten{V}^{(k,k+1)}),
\label{eq:MALSsystem}
\end{align}
where $\mat{U}_{k,k+1}$ is an $lN \times r_{k-1}(pM+1)^2r_{k+1}$ matrix. The minimal norm solution of \eqref{eq:MALSsystem} can by computed in $O(4Nl(r_{k-1}(pM+1)^2r_{k+1})^2+8(r_{k-1}(pM+1)^2r_{k+1})^3)$ flops.

Once the new super-core has been computed, it can be reshaped into an $r_{k-1}(pM+1) \times (pM+1)r_{k+1}$ matrix $\mat{V}_{k,k+1}$. The SVD of $\mat{V}_{k,k+1}$ is
\begin{align*}
\mat{V}_{k,k+1} &= \mat{U} \; \mat{S} \; \mat{V}^T,
\end{align*}
where $\mat{U},\mat{V}$ are orthogonal matrices and $\mat{S}$ is a diagonal matrix with positive entries $s_1 \geq \ldots \geq s_q$, with $q=\textrm{min}(r_{k-1}(pM+1),(pM+1)r_{k+1})$. Its computation requires $O(4(r_{k-1}^2(pM+1)^3r_{k+1})+8(r_{k-1}(pM+1)^3r_{k+1})^2+9((pM+1)^3r_{k+1}^3)$ flops. The numerical rank $r_k$ can be determined from a given tolerance $\tau$ such that $s_1 \geq \cdots \geq s_{r_k} \geq \tau \geq \cdots \geq s_q$. If the sweep is going from left to right then we update $\ten{V}^{(k)},\ten{V}^{(k+1)}$ as
\begin{align*}
\ten{V}^{(k)} &:= \textrm{reshape}(\mat{U},[r_{k-1},(pM+1),r_k]),\\
\ten{V}^{(k+1)} &:= \textrm{reshape}(\mat{S}\mat{V}^T,[r_{k},(pM+1),r_{k+1}]),
\end{align*}
which makes $\ten{V}^{(k)}$ left orthogonal. If the sweep is going from right to left then the updates are
\begin{align*}
\ten{V}^{(k)} &:= \textrm{reshape}(\mat{U}\mat{S},[r_{k-1},(pM+1),r_k]),\\
\ten{V}^{(k+1)} &:= \textrm{reshape}(\mat{V}^T,[r_{k},(pM+1),r_{k+1}]),
\end{align*}
where now $\ten{V}^{(k+1)}$ is right orthogonal. The pseudocode for the MALS algorithm is given in Algorithm \ref{alg:MALS}. A Matlab/Octave implementation of Algorithm \ref{alg:MALS} can also be freely downloaded from \url{https://github.com/kbatseli/MVMALS}.

\begin{alg}MIMO Volterra MALS Identification\\
\label{alg:MALS}
\textit{\textbf{Input}}: $N$ samples of inputs $u_1,\ldots,u_p$, outputs $y_1,\ldots,y_l$, degree $d$, memory $M$\\
\textit{\textbf{Output}}:\makebox[0pt][l]{ TN-cores $\ten{V}^{(1)},\ldots,\ten{V}^{(d)}$ that solve Problem \ref{problem:sysidproblem}}
\begin{algorithmic}
\State Initialize right orthogonal TN-cores $\ten{V}^{(1)},\ldots,\ten{V}^{(d)}$ of ranks 1
\While{termination criterion not satisfied}
\For{$i=1,\ldots,d-1$}
\State $\textrm{vec}(\ten{V}^{(i,i+1)}) \gets$ Compute and solve \eqref{eq:MALSsystem}
\State $\mat{V}_{i,i+1} \gets$ reshape($\ten{V}^{(i,i+1)},[r_{i-1}(pM+1),(pM+1)r_{i+1}])$
\State Compute the SVD of $\mat{V}_{i,i+1}$
\State Determine numerical rank $r_i$
\State $\ten{V}^{(i)} \gets$ reshape($\mat{U},[r_{i-1},(pM+1),r_i])$
\State $\ten{V}^{(i+1)} \gets$ reshape($\mat{S}\mat{V}^T,[r_i,(pM+1),r_{i+1}])$
\EndFor
\State Repeat the above loop in the reverse order
\EndWhile
\end{algorithmic}
\end{alg}
The same remarks as in the ALS case apply. The tolerance $\tau$ for the determination of the numerical rank $r_k$ can be chosen such that the error is below a certain threshold. In our implementation we opted for the default choice in Matlab/Octave, which is $\tau  = \epsilon\, s_1\, \textrm{max}(r_{i-1}(pM+1),(pM+1)r_{i+1})$, where $\epsilon$ is the machine precision. If the tolerance on the obtained solution is set too low then the MALS algorithm will have the tendency to increase the TN-ranks to very high values, resulting in higher computational complexity. However, it does not make sense in system identification to require that the solution interpolates the measured output to a high accuracy (say up to the machine precision), which is essentially overfitting.

\section{Numerical Experiments}
\label{sec:experiments}
In this section we demonstrate the two proposed identification algorithms. All computations were done on an Intel i5 quad-core processor running at 3.3 GHz with 16 GB RAM. We are not aware of any other publicly available algorithms that are able to identify high-degree MIMO Volterra systems, say, at $d=10$.
\subsection{Decaying multi-dimensional exponentials}
\label{ex:decay}
\label{ex:decayexp}
First, we demonstrate the validity of Algorithms \ref{alg:ALS} and \ref{alg:MALS} by means of an artificial SISO example. Symmetric Volterra kernels were generated up to degree $d=10$ for a fixed memory $M=7$ and containing exponentially decaying coefficients. The $i$th symmetric Volterra kernel $h_i$ contains the entries
\begin{align*}
h_i(k_1,\ldots,k_i) =\exp{(-k_1^2 - k_2^2 - \cdots -k_i^2)}.
\end{align*}
Each of the $k_i$ indices attain the values $0,0.1,0.2,\ldots,0.6$. For each degree $d$ a random input signal of 5000 samples, uniformly distributed over the interval $[0,1]$, was generated. The Volterra kernel was estimated by solving \eqref{eq:bigsystem} directly using the pseudoinverse of $\mat{U}$ where possible. The MALS algorithm was used to determine a solution for which $||\mat{y}-\mat{\hat{y}} ||_2 / ||\mat{y}||_2 < \num{1e-4}$ was satisfied. The TT-ranks determined by the MALS algorithm were then used to determine a solution with the same relative accuracy using the ALS method. Both the MALS and ALS algorithms always ran with the first 700 samples of the input and output. Table~\ref{tbl:exp1} lists the run times in seconds for the three methods, the maximal TT-rank and the total number of identified Volterra kernel elements for each degree. Starting from degree $d=5$ it was not possible anymore to obtain the Volterra tensor directly since this would require the inversion of a $32768 \times 32768$ matrix. From $d=3$, the maximal TT-rank stabilizes to 8 for all TT-cores. Even though the ALS method exhibits low computational complexity, its convergence is much slower than MALS, resulting in larger run times. The obtained solutions were validated by simulating the output using the 4300 remaining input points. Fig.~\ref{fig:ex1validation} shows both the real and simulated output from the MALS solution for samples 1640 up to 1700. The two output signals are almost indistinguishable. No difference between the ALS and MALS solutions could be observed.

\begin{table}[ht]
\begin{center}
\caption{Run times, maximal TT-rank and number of estimated Volterra tensor elements for an increasing degree $d$.}
\begin{tabular}{@{}rrrrrr@{}}
$d$ & \multicolumn{3}{c}{Run time [seconds]} & max TT-rank & $(pM+1)^d$\\ \midrule
& $\mat{U}^{\dagger}\mat{y}$ & ALS & MALS &  & \\\midrule
2 &  0.017& 0.240 & 0.077 & 6 & 64\\
3 &  0.253& 0.223 & 0.251 & 8 & 512\\
4 &  39.36& 1.771 & 0.415 & 8 & 4096\\
5 &  NA& 4.288& 0.587 & 8 & 32768\\
6 &  NA& 3.065& 0.791 & 8 & 262144\\
7 &  NA& 6.783& 0.961 & 8 & 2097152\\
8 &  NA& 13.11& 1.199 & 8 & 16777216\\
9 &  NA& 14.42& 1.384 & 8 & 134217728\\
10&  NA& 18.37& 1.576 & 8 &\num{1.0737e+9}
\end{tabular}
\label{tbl:exp1}
\end{center}
\end{table}

\begin{figure}[t]
\centering 
\includegraphics[width=3.5in]{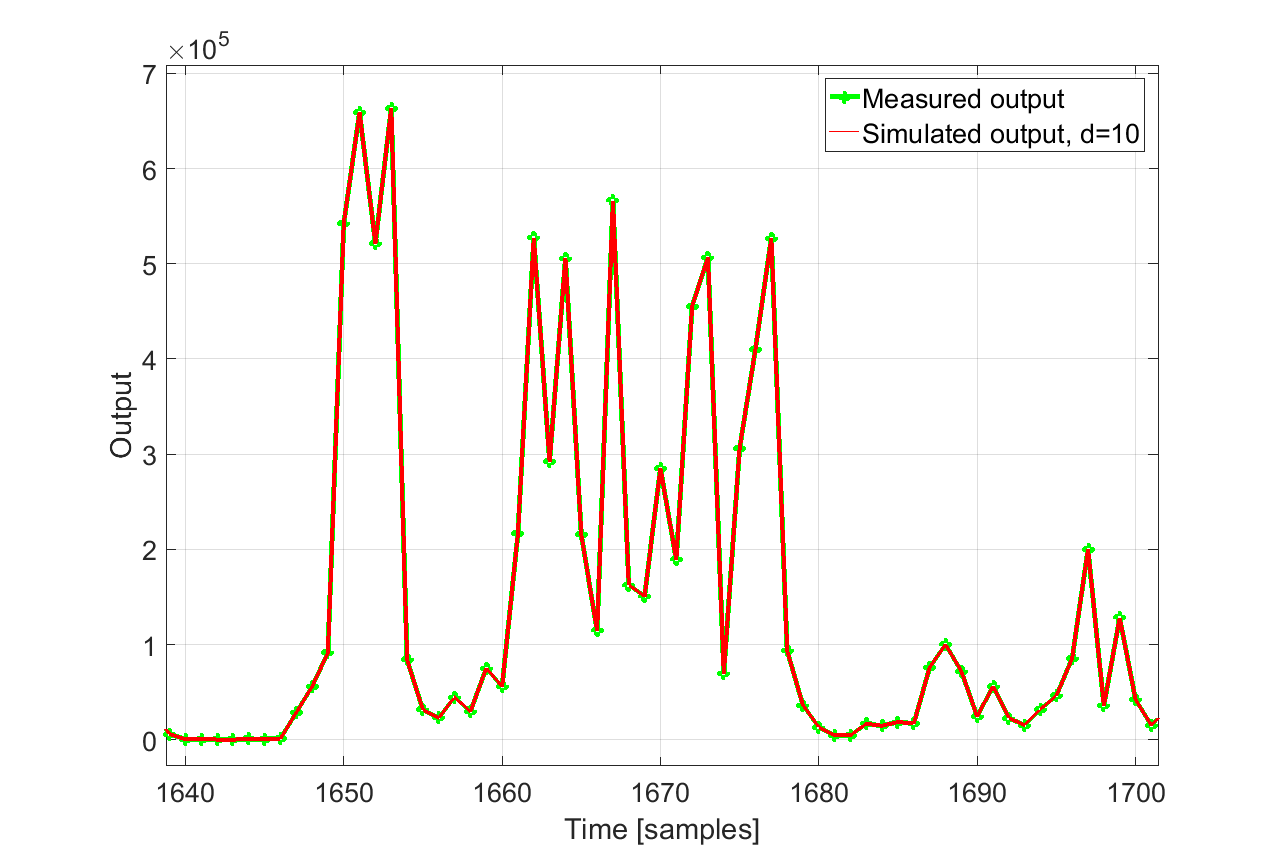}
\caption{Real and simulated output from the MALS solution with $d=10,M=7$.}
\label{fig:ex1validation}
\end{figure}

\subsection{Double balanced mixer}
In this example we consider a double balanced mixer used for upconversion. The output radio-frequency (RF) signal is determined by a 100Hz sine low-frequency (LO) signal and a 300Hz square-wave intermediate-frequency (IF) signal. A phase difference of $\pi/8$ is present between the LO and IF signals. All time series were sampled at 5 kHz for 1 second. We investigate the effect of additive output noise on the identified models. We define 5 different noise levels which are added to the measured RF output, generating signals with signal-to-noise (SNR) ratios ranging from 11dB up to 25dB. The first 700 samples of the inputs and the noisy output are then used to identify an $M=2,d=11$ two-input one-output Volterra system using the ALS method. The Volterra kernel consists in this case of 9765625 entries. The TT-ranks are all fixed to $2M+1=5$. The identified models were then used to simulate the remaining 4300 samples of the output. The SNR of the simulated output was computed by comparing the simulated output with the original noiseless output. Table~\ref{table:mixerALS} lists the SNR of the signals used in the identification (ID SNR), the relative residual of the simulated output, the run time of the identification in seconds and the SNR of the simulated signal (SIM SNR). As expected, a gradual improvement of the relative residual can be seen as the SNR of the signals used for identification increases. Although the residual remains high throughout the different SNR levels, the SNR of the simulated output is much better, with a consistent increase of 11dB. The run time varies between 2 and 6 seconds. Fig.~\ref{fig:ex2sim} shows the simulated output on the validation data for three Volterra models identified under three different SNR levels (11dB, 16dB, 25 dB). 

Next, the MALS algorithm was run on the same data to also identify $M=2,d=11$ Volterra models. Based on the ALS results we set the tolerance on the relative residual to $0.5$. This resulted in all TT-ranks being 5 for all cases. Table~\ref{table:mixerMALS} lists the SNR of the signals used in the identification (ID SNR), the relative residual of the simulated output, the run time of the identification in seconds and the SNR of the simulated signal (SIM SNR). The relative residuals and SNRs of the simulated output are very close to the results obtained by the ALS identification. Just as in Example \ref{ex:decay}, MALS is able to finish the identification faster than the ALS method. Furthermore, the run time does not vary as much as in the ALS case. Lowering the tolerance for the data with low SNR resulted in the MALS method increasing the TT ranks significantly, up to the point that the 5000 samples were not sufficient anymore. This indicates a tendency of the MALS method to overfit. 
\begin{table}[ht]
\centering
\caption{ALS identification for 5 different SNR levels.}
\label{table:mixerALS} 
\begin{tabular}{@{}l|rrrrr@{}}
ID SNR  & 11dB & 13dB & 16dB & 19dB & 25dB \\ \midrule
$\frac{||\mat{y}-\mat{\hat{y}} ||_2}{||\mat{y}||_2}$ &$.255$ & $.208$ & $.151$& $.105$ & $.052$\\ \midrule
Run time & 2.3s & 5.3s & 6.4s & 3.2s & 2.2s\\ \midrule
SIM SNR  & 22dB & 24dB & 27dB & 30dB & 37dB \\
\end{tabular}
\end{table}

\begin{figure}[t]
\centering 
\includegraphics[width=3.5in]{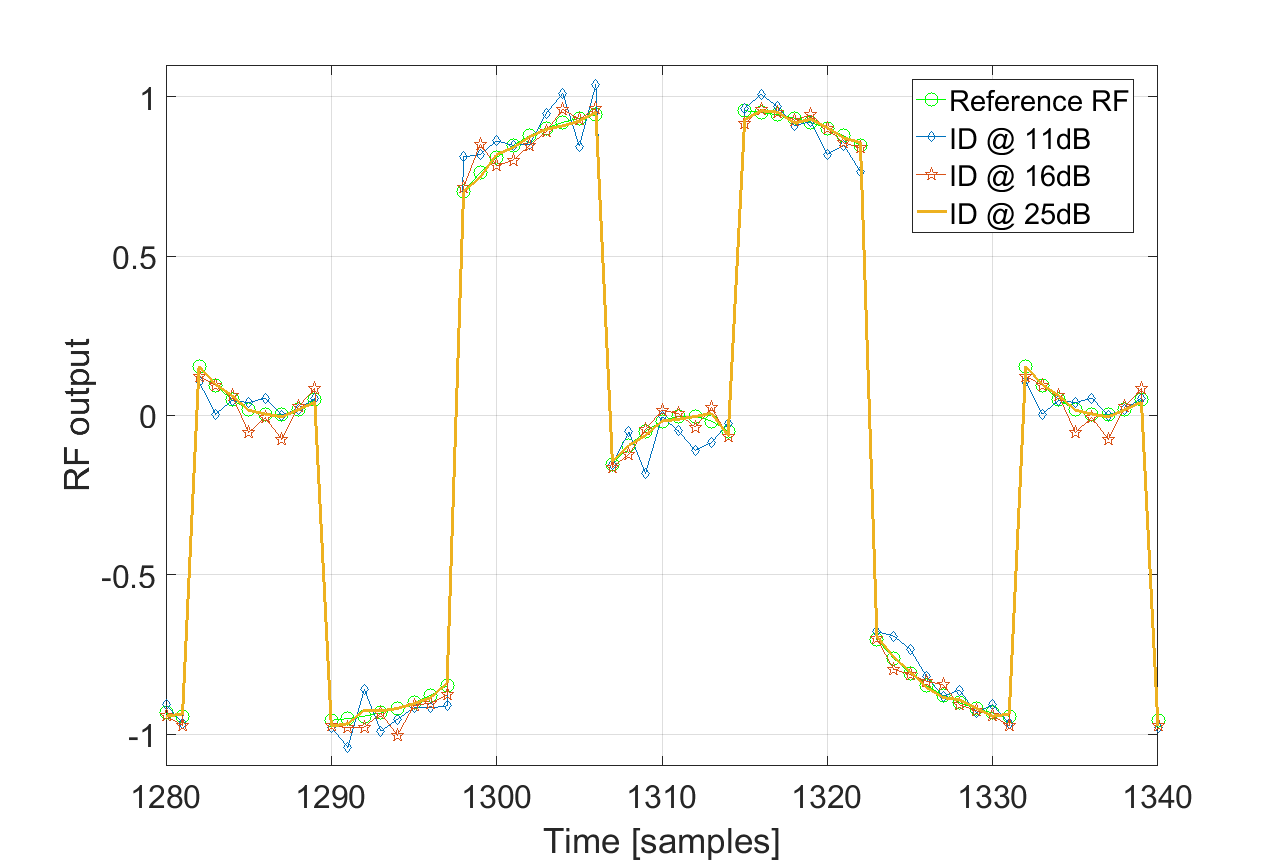}
\caption{Real and simulated output from the ALS identified models under different SNRs.}
\label{fig:ex2sim}
\end{figure}

\begin{table}[ht]
\centering
\caption{MALS identification for 5 different SNR levels with tolerance of $0.5$.}
\label{table:mixerMALS} 
\begin{tabular}{@{}l|rrrrr@{}}
ID SNR  & 11dB & 13dB & 16dB & 19dB & 25dB \\ \midrule
$\frac{||\mat{y}-\mat{\hat{y}} ||_2}{||\mat{y}||_2}$ &$.251$ & $.214$ & $.152$& $.105$ & $.055$\\ \midrule
Run time & 1.3s & 1.1s & 1.1s & 1.1s & 1.1s\\ \midrule
SIM SNR  & 24dB & 25dB & 28dB & 31dB & 37dB \\
\end{tabular}
\end{table}

\section{Conclusions}
This article presented two new and remarkably efficient identification algorithms for high-order MIMO Volterra systems. The identification problem was rephrased in terms of the Volterra tensor, which is never explicitly constructed but instead always stored in the highly economic TT format. Both proposed identification algorithms are iterative, starting from an initial orthogonal guess for the TT-cores and updating them until a desired accuracy is acquired. The algorithms are guaranteed to monotonically converge and numerical stability is ensured by retaining orthogonality in the TT-cores. The efficiency of both identification algorithms was demonstrated by numerical examples, where reliable MIMO Volterra systems of degrees 10 were estimated in only a few seconds. Even though its computational complexity is lower, the ALS method was found to converge slower than the MALS algorithm when producing solutions with the same accuracy. The MALS method showed a tendency to increase the TT-ranks under the presence of high noise levels. Extending the robustness of these methods to noisy data will be the subject of further research.

\bibliographystyle{plain}        
\bibliography{references}           



\end{document}